\documentclass[leqno,12pt]{amsart}
\usepackage{amsmath,amstext,amssymb,amsopn,amsthm,mathrsfs}
\usepackage{epsfig,graphics,latexsym,graphicx,color}
\definecolor{darkgreen}{rgb}{0,.5,0}

\numberwithin{equation}{section}

\allowdisplaybreaks

\newtheorem{theorem}{Theorem}[section]

\newtheorem{lemma}{Lemma}[section]
\newtheorem*{rem*}{Remark}

\begin{document}
\footnotetext{
\emph{2010 Mathematics Subject Classification.} Primary 42B20, 47B07; Secondary: 42B25,47G99.

\emph{Key words and phrases.} Bilinear Fractional Integral Operators; Characterization; Compactness; Iterated Commutator}

\title[]{Characterization of CMO via compactness of the commutators of bilinear fractional integral operators}

\author[]{Dinghuai Wang, Jiang Zhou$^\ast$ and Wenyi Chen}
\address{College of Mathematics and System Sciences \endgraf
         Xinjiang University \endgraf
         Urumqi 830046 \endgraf
         Republic of China}
\email{Wangdh1990@126.com; zhoujiangshuxue@126.com}
\address{School of Mathematics and Statistics \endgraf
         WuHan University \endgraf
         WuHan 430072 \endgraf
         Republic of China}
\email{wychencn@hotmail.com}
\thanks{The research was supported by National Natural Science Foundation
of China (Grant No.11661075 and No.11261055). \\ \qquad * Corresponding author, Email: zhoujiangshuxue@126.com.}

\begin{abstract}
Let $I_{\alpha}$ be the bilinear fractional integral operator, $B_{\alpha}$ be a more singular family of bilinear fractional integral operators and $\vec{b}=(b,b)$. B\'{e}nyi et al. in \cite{B1} showed that if $b\in {\rm CMO}$, the {\rm BMO}-closure of $C^{\infty}_{c}(\mathbb{R}^n)$, the commutator $[b,B_{\alpha}]_{i}(i=1,2)$ is a separately compact operator. In this paper, it is proved that $b\in {\rm CMO}$ is necessary for $[b,B_{\alpha}]_{i}(i=1,2)$ is a compact operator. Also, the authors characterize the compactness of the {\bf iterated} commutator $[\Pi\vec{b},I_{\alpha}]$ of bilinear fractional integral operator. More precisely, the commutator $[\Pi\vec{b},I_{\alpha}]$ is a compact operator if and only if $b\in {\rm CMO}$.
\end{abstract}
\maketitle

\maketitle

\section{Introduction}

A locally integrable function $f$ is said to belong to {\rm BMO} space if there exists a constant
$C > 0$ such that for any cube $Q\subset \mathbb{R}^n$,
$$\frac{1}{|Q|}\int_{Q}|f(x)-f_{Q}|dx\leq C,$$
where $f_{Q}=\frac{1}{|Q|}\int_{Q}f(x)dx$ and the minimal constant $C$ is defined by $\|f\|_{*}$. There are a number of classical results that demonstrate ${\rm BMO}$ functions are the right collections to do harmonic analysis on the boundedness of commutators. A well known result of Coifman, Rochberg and Weiss \cite{CRW} states that the commutator
$$[b,T](f)=bT(f)-T(bf)$$
is bounded on some $L^p$, $1<p<\infty$, if and only if $b\in \mathrm{BMO}$, where $T$ be the classical Calder\'{o}n-Zygmund operator. In 1978, Uchiyama \cite{U} refined the boundednss results on the commutator to compactness. This is a achieved by requiring the commutator with symbol to be in ${\rm CMO}$, which is the closure in ${\rm BMO}$ of the space of $C^{\infty}$ functions with compact support. In recent years, the compactness of commutators has been extensively studied already, as Chen, Ding and Wang \cite{CDW1}, \cite{CDW2} and Wang \cite{W}. The interest in the compactness of $[b, T]$ in complex analysis is from the connection between the commutators and the Hankel-type operators. In fact, the authors of \cite{KL1} and \cite{KL2} have applied commutator theory to give a compactness characterization of Hankel operators on holomorphic Hardy spaces $H^{2}(D)$, where $D$ is a bounded, strictly pseudoconvex domain in $\mathbb{C}^n$. It is perhaps for this important reason that the compactness of $[b, T]$ attracted one¡¯s attention among researchers in PDEs.

In the multilinear setting, the boundedness results for commutators with symbols in {\rm BMO} started to receive attention only a few years ago, see \cite{LOPTT}, \cite{P}, \cite{PT} or \cite{T}. Compactness results in the multilinear setting have just began to be studied. B\'{e}nyi and Torres \cite{BT}, B\'{e}nyi et al. \cite{B1} and \cite{B2} showed that symbols in {\rm CMO} again produce compact commutators. Specially, B\'{e}nyi et al. in \cite{B1} showed that if $b\in {\rm CMO}$, the commutator $[b,B_{\alpha}]_{i}(i=1,2)$ is a separately compact operator. More precisely, it is obtained that if $b\in {\rm CMO}$ and $g\in L^{p_{2}}$ is fixed, $[b,B_{\alpha}]_{i}(\cdot,g)(i=1,2)$ is a compact operator from $L^{p_{1}}$ to $L^{q}$. Unfortunately, it is unknown that if $b\in {\rm CMO}$, are the commutators $[b,B_{\alpha}]_{i}, i = 1,2$, jointly compact? We intend to study this question in future work, however, in this paper, we first give the necessary condition for commutators $[b,B_{\alpha}]_{i}$ are jointly compact.

Another subject of this paper is to consider the characterization of compactness of the iterated commutator of $I_{\alpha}$. In 2015, Chaffee and Torres \cite{CT} characterized the compactness of the {\bf linear commutators} of bilinear fractional integral operators acting on product of Lebesgue spaces. In this paper, the characterization of compactness of the {\bf iterated commutators} will be considered.

\vspace{0.5cm}

To state the main result of this paper, we first recall some necessary notions and notation.

It is well known that the fractional integral $\mathcal{I}_{\alpha}$ of order $\alpha (0<\alpha<n)$ plays an
important role in harmonic analysis, PDE and potential theory (see \cite{S}). Recall that
$\mathcal{I}_{\alpha}$ is defined by
$$\mathcal{I}_{\alpha}f(x)=\int_{\mathbb{R}^n}\frac{f(y)}{|x-y|^{n-\alpha}}dy.$$
For the bilinear case, the bilinear fractional integral operator $I_{\alpha}$, $0<\alpha<2n$, is defined by
$$I_{\alpha}(f_{1},f_{2})(x)=\int_{\mathbb{R}^n}\int_{\mathbb{R}^n}\frac{f_{1}(y_{1})f_{2}(y_{2})}{\big(|x-y_{1}|+|x-y_{2}|\big)^{2n-\alpha}}dy_{1}dy_{2}.$$
In this paper, we will consider the following equivalent operator
$$I_{\alpha}(f_{1},f_{2})(x)=\int_{\mathbb{R}^n}\int_{\mathbb{R}^n}\frac{f_{1}(y_{1})f_{2}(y_{2})}{\big(|x-y_{1}|^{2}+|x-y_{2}|^{2}\big)^{n-\alpha/2}}dy_{1}dy_{2}.$$
Its iterated commutators with $\vec{b}=(b_{1},b_{2})$ is given by
$$I_{\alpha}(f_{1},f_{2})(x)=\int_{\mathbb{R}^n}\int_{\mathbb{R}^n}\frac{( b_{1}(x)-b_{1}(y_{1}))( b_{2}(x)-b_{2}(y_{2}))f_{1}(y_{1})f_{2}(y_{2})}{\big(|x-y_{1}|^{2}+|x-y_{2}|^{2}\big)^{n-\alpha/2}}dy_{1}dy_{2}.$$

We will now examine a more singular family of bilinear fractional integral operators,
$$B_{\alpha}(f,g)(x)=\int_{\mathbb{R}^n}\frac{f(x-y)g(x+y)}{|y|^{n-\alpha}}dy.$$
This operator was first introduced by Grafakos in \cite{G}, and later studied by Grafakos and Kalton \cite{GK} and Kenig and Stein \cite{KS}. The commutators $[b,B_{\alpha}]_{i}$ of $B_{\alpha}$ with $b$ can be written as
\begin{eqnarray*}
[b,B_{\alpha}]_{1}(f,g)(x)&=&bB_{\alpha}(f,g)-B_{\alpha}(bf,g)\\
&=&\int_{\mathbb{R}^n}\frac{b(x)-b(y)}{|x-y|^{n-\alpha}}f(y)g(2x-y)dy,
\end{eqnarray*}
the definition of $[b,B_{\alpha}]_{2}$ be the similar as $[b,B_{\alpha}]_{1}$. In what follows, we need only consider one of these two commutators.

\vspace{0.5cm}

For $1<p\leq q<\infty$, recall that the Muckenhoupt class of weights consists of all nonnegative, locally integrable functions $\omega$ such that
$$\sup_{Q}\bigg(\frac{1}{|Q|}\int_{Q}\omega^{q}(x)dx\bigg)\bigg(\frac{1}{|Q|}\int_{Q}\omega(x)^{-p'}dx\bigg)^{q/p'}<\infty.$$
We also recall the definition of the multiple or vector weights used in the bilinear setting. For $1<p_{1},p_{2}<\infty, {\bf P}=(p_{1},p_{2}), 0<\alpha<2n, \frac{\alpha}{n}<\frac{1}{p_{1}}+\frac{1}{p_{2}}$, and $q$ such that $\frac{1}{q}=\frac{1}{p_{1}}+\frac{1}{p_{2}}-\frac{\alpha}{n}$, a vector weight ${\bf\omega}= (\omega_{1},\omega_{2})$ belongs to ${\bf A_{P,q}}$ if
$$\sup_{Q}\bigg(\frac{1}{|Q|}\int_{Q}\mu_{\mathbf{\omega}}(x)dx\bigg)
\bigg(\frac{1}{|Q|}\int_{Q}\omega_{1}(x)^{-p'_{1}}dx\bigg)^{q/p'_{1}}\bigg(\frac{1}{|Q|}\int_{Q}\omega_{2}(x)^{-p'_{2}}dx\bigg)^{q/p'_{2}}<\infty.$$
where the notation $\mu_{\mathbf{\omega}}=\omega_{1}^{q}\omega_{2}^{q}$. It was shown by Moen in \cite{M} that if ${\bf \omega}\in \mathbf{A_{P,q}}$ then $\omega_{i}^{-p'_{i}}\in A_{2p'_{i}}$ and $\mu_{\omega}\in A_{2q}$. In addition, the weights in $\mathbf{A_{P,q}}$ are precisely those for which
$$I_{\alpha}: L^{p_{1}}(\omega^{p_{1}}_{1})\times L^{p_{2}}(\omega^{p_{2}}_{2})\rightarrow L^{q}(\mu_{\omega})$$
is bounded.

\vspace{0.5cm}

Now we return to our main results.
\begin{theorem}\label{main1}
Let $1<p,p_{1},p_{2},q<\infty, 0<\alpha<n$ such that $\frac{1}{p}=\frac{1}{p_{1}}+\frac{1}{p_{2}}$ and $\frac{1}{q}=\frac{1}{p_{1}}+\frac{1}{p_{2}}-\frac{\alpha}{n}$. If $[b,B_{\alpha}]_{1}$ is a compact operator from $L^{p_{1}}\times L^{p_{2}}$ to $L^{q}$, then $b\in{\rm CMO}$.
\end{theorem}

\begin{theorem}\label{main2}
Let $1<p,p_{1},p_{2},q<\infty, 0<\alpha<2n$ such that $\frac{1}{p}=\frac{1}{p_{1}}+\frac{1}{p_{2}}$,
$\frac{\alpha}{n}<\frac{1}{p_{1}}+\frac{1}{p_{2}}$ and $\frac{1}{q}=\frac{1}{p_{1}}+\frac{1}{p_{2}}-\frac{\alpha}{n}$. For the local integral function $b$ and $\vec{b}=(b,b)$, the following are equivalent,
\begin{enumerate}
\item [\rm(A1)] $b\in {\rm CMO}$.
\item [\rm(A2)] $[\Pi\vec{b},I_{\alpha}]: L^{p_{1}}(\omega_{1}^{p_{1}})\times L^{p_{2}}(\omega_{2}^{p_{2}})\rightarrow L^{q}(\omega_{1}^{q}\omega_{2}^{q})$ is a compact operator for all $\mathbf{\omega}=(\omega_{1},\omega_{2})$ such that $\omega_{1}^{\frac{p_{1}q}{p}},\omega_{2}^{\frac{p_{2}q}{p}}\in A_{p}$.
\item [\rm(A3)] $[\Pi\vec{b},I_{\alpha}]: L^{p_{1}}\times L^{p_{2}}\rightarrow L^{q}$ is a compact operator.
\end{enumerate}
\end{theorem}

\section{Main lemmas}

As mentioned in the introduction, ${\rm CMO}$ is the closure in ${\rm BMO}$ of the space of $C^{\infty}$ functions with compact support. In \cite{U}, it was shown that ${\rm CMO}$ can be characterized in the following way.
\begin{lemma}(\cite{U})\label{lem1}
Let $b\in {\rm BMO}$. Then $b$ is in ${\rm CMO}$ if and only if

\begin{equation}\label{2.1.1}
\lim_{a\rightarrow 0}\sup_{|Q|=a}\frac{1}{|Q|}\int_{Q}|b(x)-b_{Q}|dx=0;
\end{equation}

\begin{equation}\label{2.1.2}
\lim_{a\rightarrow \infty}\sup_{|Q|=a}\frac{1}{|Q|}\int_{Q}|b(x)-b_{Q}|dx=0;
\end{equation}

\begin{equation}\label{2.1.3}
\lim_{|y|\rightarrow 0}\frac{1}{|Q|}\int_{Q}|b(x+y)-b_{Q}|dx=0,  \ for\ each ~Q.
\end{equation}
\end{lemma}

To prove Theorem \ref{main1} and Theorem \ref{main2}, we need the following results.

\begin{lemma}\label{lem2}
Support that $b\in {\rm BMO}$ with $\|b\|_{*}=1$. If for some $0<\epsilon<1$ and a cube $Q$ with its center at $x_{Q}$ and $r_{Q}$, $b$ is not a constant on cube $Q$ and satisfies
\begin{equation}\label{2.2.1}
\frac{1}{|Q|}\int_{Q}|b(y)-b_{Q}|dy>\epsilon,
\end{equation}
then for the functions $g(y)=\frac{|Q|\chi_{(2Q)^{c}}(y)}{|y-x_{Q}|^{n/p_{2}+n}}$ and $f$ is defined by
\begin{equation}\label{2.2.2}
f(y)=|Q|^{-1/p_{1}}\big(\mathrm{sgn}(b(y)-b_{Q})\big)\chi_{Q}(y),
\end{equation}
There exists constants $\gamma_{1},\gamma_{2},\gamma_{3}$ satisfying $\gamma_{2}>\gamma_{1}>2$ and $\gamma_{3}>0$, such that
\begin{equation}\label{2.2.3}
\int_{\gamma_{1}r_{Q}<|x-x_{Q}|<\gamma_{2}r_{Q}}\big|[b,B_{\alpha}]_{1}(f,g)(x)\big|^{q}dx\geq \gamma_{3}^{q},
\end{equation}
\begin{equation}\label{2.2.4}
\int_{|x-x_{Q}|>\gamma_{2}r_{Q}}\big|[b,B_{\alpha}]_{1}(f,g)(x)\big|^{q}dx\leq \frac{\gamma_{3}^{q}}{4^{q}}.
\end{equation}

Moreover, there exists a constant $0<\beta<< \gamma_{2}$ depending only on $p_{1},p_{2},n$ such that  for all measurable subsets $E\subset \big\{x:\gamma_{1}r_{Q}<|x-x_{Q}|<\gamma_{2}r_{Q}\big\}$ satisfying $\frac{|E|}{|Q|}<\beta^{n}$, we have
\begin{equation}\label{2.2.5}
\int_{E}\big|[b,B_{\alpha}]_{1}(f,g)(x)\big|^{q}dx\leq \frac{\gamma_{3}^{q}}{4^{q}}.
\end{equation}
\end{lemma}
\begin{proof}
It is easy to check that $f$ satisfies
$$\mathrm{supp} f\subset Q,$$
$$f(y)(b(y)-b_{Q})=|Q|^{-1/p_{1}}|b(y)-b_{Q}|\chi_{Q}(y)\geq 0,$$
$$|f(y)|\leq |Q|^{-1/p_{1}},$$
$$\|f\|_{L^{p_{1}}}\leq 1,$$
$$\int \big(b(y)-b_{Q}\big)f(y)dy=|Q|^{-1/p_{1}}\int_{Q}|b(y)-b_{Q}|dy,$$
and $g$ satisfies that $\|g\|_{L^{p_{2}}}=C$ and for $x\in (2nQ)^{c},y\in Q$, we get
$$|2x-y-y|=2|x-y|\geq 2|x-x_{Q}|-2|y-x_{Q}|\geq 2nr_{Q}-\sqrt{n}r_{Q}\geq nr_{Q},$$
which implies that $2x-y\in \big(Q(y, 2\sqrt{n}r_{Q})\big)^{c}\subset (2Q)^{c}$ and $g(2x-y)\approx |Q|\cdot|x-x_{Q}|^{-n/p_{2}-n}$.

We first establish the following several technical estimates. For a cube $Q$ with center $x_{Q}$ and satisfying (\ref{2.2.1}) for some $\epsilon>0$ and $x\in (2nQ)^{c}$, the following point-wise estimates hold:
\begin{equation}\label{2.2.6}
|B_{\alpha}\big((b-b_{Q})f,g\big)(x)|\lesssim |Q|^{\frac{1}{p'_{1}}+1}|x-x_{Q}|^{-2n-n/p_{2}+\alpha},
\end{equation}
\begin{equation}\label{2.2.7}
|B_{\alpha}\big((b-b_{Q})f,g\big)(x)|\gtrsim \epsilon |Q|^{\frac{1}{p'_{1}}+1}|x-x_{Q}|^{-2n-n/p_{2}+\alpha},
\end{equation}
\begin{equation}\label{2.2.8}
|B_{\alpha}\big(f,g\big)(x)|\lesssim |Q|^{\frac{1}{p'_{1}}+1}|x-x_{Q}|^{-2n-n/p_{2}+\alpha},
\end{equation}
where $f,g$ as above and the constants involved are independent of $b, f, g$ and $\epsilon$.

To prove (\ref{2.2.6}), from the fact that $\|b\|_{*}=1$ and $x\in (2nQ)^{c}$, we have
\begin{eqnarray*}
|B_{\alpha}\big((b-b_{Q})f,g\big)(x)|&=&\bigg|\int_{Q}\frac{(b(y)-b_{Q})f(y)g(2x-y)}
{|x-y|^{n-\alpha}}dy\bigg|\\
&\lesssim& |Q|^{\frac{1}{p'_{1}}}|x-x_{Q}|^{-2n-n/p_{2}+\alpha}
\int_{Q}(b(y)-b_{Q})f(y)dy\\
&\lesssim& |Q|^{\frac{1}{p'_{1}}+1}|x-x_{Q}|^{-2n-n/p_{2}+\alpha}.
\end{eqnarray*}

For (\ref{2.2.7}), by $x\in (2nQ)^{c}$ and $y\in Q$, we have $|x-y|\approx |x-x_{Q}|$. Using that $\big(b(y)-b_{Q}\big)f(y)\geq 0$, we can compute
\begin{eqnarray*}
|B_{\alpha}\big((b-b_{Q})f,g\big)(x)|&=&\bigg|\int_{Q}\frac{(b(y)-b_{Q})f(y)g(2x-y)}{|x-y|^{n-\alpha}}dy\bigg|\\
&\gtrsim& |Q||x-x_{Q}|^{-2n-n/p_{2}+\alpha}\int_{Q}\big(b(y)-b_{Q}\big)f(y)dy\\
&=&|Q|^{1/p'_{1}}|x-x_{Q}|^{-2n-n/p_{2}+\alpha}\int_{Q}\big|b(y)-b_{Q}\big|dy\\
&\gtrsim& \epsilon |Q|^{\frac{1}{p'_{1}}+1}|x-x_{Q}|^{-2n-n/p_{2}+\alpha}.
\end{eqnarray*}
Finally using that $|f(y)|\leq |Q|^{-1/p_{1}}$ we obtain (\ref{2.3.8}) as follows.
\begin{eqnarray*}
|B_{\alpha}\big(f,g\big)(x)|=\bigg|\int_{Q}\frac{f(y)g(2x-y)}
{|x-y|^{n-\alpha}}dy\bigg|\lesssim |Q|^{\frac{1}{p'_{1}}+1}|x-x_{Q}|^{-2n-n/p_{2}+\alpha}.
\end{eqnarray*}

\vspace{0.5cm}

Now, we give the proofs of (\ref{2.2.3})-(\ref{2.2.4}). Note that for $b\in {\rm BMO}$, we have
$$\bigg(\int_{2^{s}r_{Q}<|x-x_{Q}|<2^{s+1}d_{j}}|b(x)-b_{Q}|^{q}dx\bigg)^{1/q}\lesssim s2^{sn/q}|Q|^{1/q}.$$

Taking $\nu>16,$ by (\ref{2.2.6}) we obtain
\begin{eqnarray*}
&&\bigg(\int_{|x-x_{Q}|>\nu r_{Q}}\big|(b(x)-b_{Q})B_{\alpha}(f,g)(x)\big|^{q}dx\bigg)^{\frac{1}{q}}\\
&&\leq C|Q|^{\frac{1}{p'_{1}}+1}\sum_{s=\lfloor\log_{2}\nu\rfloor}^{\infty}
\bigg(\int_{2^{s}r_{Q}<|x-x_{Q}|<2^{s+1}r_{Q}}\frac{|b(x)-b_{Q}|^{q}}{|x-x_{Q}|^{q(2n-\alpha+n/p_{2})}}dx\bigg)^{\frac{1}{q}}\\
&&\leq C|Q|^{\frac{1}{p'_{1}}+1}\sum_{s=\lfloor\log_{2}\nu\rfloor}^{\infty}2^{-s(2n-\alpha+n/p_{2})}|Q|^{-2+\frac{\alpha}{n}-\frac{1}{p_{2}}}
\bigg(\int_{2^{s}r_{Q}<|x-x_{Q}|<2^{s+1}d_{j}}|b(x)-b_{Q}|^{q}dx\bigg)^{\frac{1}{q}}\\
&&\leq C\sum_{s=\lfloor\log_{2}\nu\rfloor}^{\infty}s2^{-s(2n-\alpha+\frac{n}{p_{2}}-\frac{n}{q})}\\
&&\leq C\sum_{s=\lfloor\log_{2}\nu\rfloor}^{\infty}2^{-s(2n-\alpha+\frac{n}{p_{2}}-\frac{n}{q}-\frac{1}{2})}\\
&&\leq C\nu^{-(2n-\alpha+\frac{n}{p_{2}}-\frac{n}{q}-\frac{1}{2})},
\end{eqnarray*}
where we have used that $2n-\alpha+\frac{n}{p_{2}}-\frac{n}{q}-\frac{1}{2}>0$ and $s\leq 2^{s/2}$ for $4\leq \lfloor\log_{2}\nu\rfloor\leq s$.

For $\mu>\nu,$ using (\ref{2.2.7}) and the estimates above, we get
\begin{eqnarray*}
&&\int_{\nu r_{Q}<|x-x_{Q}|<\mu r_{Q}}\big|[b,B_{\alpha}](f,g)(x)\big|^{q}dx\\
&&\geq C\bigg(\big|B_{\alpha}((b-b_{Q})f,g)(x)\big|^{q}dx\bigg)^{1/q}\\
&&\qquad -C\bigg(\int_{\nu r_{Q}<|x-x_{Q}|}\big|(b(x)-b_{Q})B_{\alpha}(f,g)(x)\big|^{q}dx\bigg)^{1/q}\\
&&\geq C\epsilon|Q|^{\frac{1}{p'_{1}}+1}\bigg(\int_{\nu r_{Q}<|x-x_{Q}|<\mu r_{Q}}|x-x_{Q}|^{q(-2n-n/p_{2}+\alpha)}dx\bigg)^{1/q}-C\nu^{-(2n-\alpha+\frac{n}{p_{2}}-\frac{n}{q}-\frac{1}{2})}\\
&&\geq C\epsilon \Big(\nu^{-2nq+nq/p_{2}+\alpha q}-\mu^{-2nq+nq/p_{2}+\alpha q}\Big)^{\frac{1}{q}}-C\nu^{-(2n-\alpha+\frac{n}{p_{2}}-\frac{n}{q}-\frac{1}{2})}.
\end{eqnarray*}
Once again, the constants appearing above are independent of $Q$. It is easy to see that we can select $\gamma_{1},\gamma_{2}$ in place of $\nu,\mu$ with $\gamma_{2}>>\gamma_{1}$, then (\ref{2.3.1}) and (\ref{2.3.2}) are verified for some $\gamma_{3}>0$.

We now verified (\ref{2.2.5}). Let $E\subset \big\{\gamma_{1}r_{Q}<|x-x_{Q}|<\gamma_{2}r_{Q}\big\}$ be an arbitrary measurable set. It follows from Minkowski inequality that
\begin{eqnarray*}
&&\bigg(\int_{E}\big|[b,B_{\alpha}]_{1}(f,g)(x)\big|^{q}dx\bigg)^{\frac{1}{q}}\\
&&\lesssim |Q|^{\frac{1}{p'_{1}}+1}\bigg(\int_{E}|x-x_{Q}|^{-q(2n+n/p_{2}-\alpha)}dx\bigg)^{\frac{1}{q}}
+|Q|^{\frac{1}{p'_{1}}+1}\bigg(\int_{E}\frac{|b(x)-b_{Q}|^{q}}{|x-x_{Q}|^{q(2n-\alpha+n/p_{2})}}dx\bigg)^{\frac{1}{q}}\\
&&\lesssim \bigg(\frac{|E|^{1/q}}{|Q|^{1/q}}+\Big(\frac{1}{|Q|}\int_{E}|b(x)-b_{Q}|^{q}dx\Big)^{\frac{1}{q}}\bigg)
\end{eqnarray*}
It is proved in \cite[p.269]{CDW2} that
$$\frac{1}{|Q|}\int_{E}|b(x)-b_{Q}|^{q}dx\lesssim \frac{|E|}{|Q|}\bigg(1+\log \Big(\frac{\tilde{C}|Q|}{|E|}\Big)\bigg)^{[q]+1}.$$
Taking $0<\beta<\min\{\tilde{C}^{1/n},\gamma_{2}\}$ and sufficiently small, then (\ref{2.2.5}) holds.
\end{proof}

\begin{lemma}\label{lem3}
Support that $b\in {\rm BMO}$ with $\|b\|_{*}=1$. If for some $0<\epsilon<1$ and a cube $Q$ with its center at $x_{Q}$ and $r_{Q}$, $b$ is not a constant on cube $Q$ and satisfies
$$\frac{1}{|Q|}\int_{Q}|b(y)-b_{Q}|dy>\epsilon^{1/2},$$
then for the function $f_{i}(i=1,2)$ defined by
\begin{equation}\label{2.3.0}
f_{i}(y_{i})=|Q|^{-1/p_{i}}\big(sgn(b(y_{i})-b_{Q})-c_{0}\big)\chi_{Q}(y_{i}),
\end{equation}
where $c_{0}=|Q|^{-1}\int_{Q}sgn\big(b(y)-b_{Q}\big)dy_{i}$ and $i=1,2$. There exists constants $\gamma_{1},\gamma_{2},\gamma_{3}$ satisfying $\gamma_{2}>\gamma_{1}>2$ and $\gamma_{3}>0$, such that
\begin{equation}\label{2.3.1}
\int_{\gamma_{1}r_{Q}<|x-x_{Q}|<\gamma_{2}r_{Q}}\big|[\Pi\vec{b},I_{\alpha}](f_{1},f_{2})(x)\big|^{q}dx\geq \gamma_{3}^{q},
\end{equation}
\begin{equation}\label{2.3.2}
\int_{|x-x_{Q}|>\gamma_{2}r_{Q}}\big|[\Pi\vec{b},I_{\alpha}](f_{1},f_{2})(x)\big|^{q}dx\leq \frac{\gamma_{3}^{q}}{4^{q}}.
\end{equation}

Moreover, there exists a constant $0<\beta<< \gamma_{2}$ depending only on $p_{1},p_{2},n$ such that  for all measurable subsets $E\subset \big\{x:\gamma_{1}r_{Q}<|x-x_{Q}|<\gamma_{2}r_{Q}\big\}$ satisfying $\frac{|E|}{|Q|}<\beta^{n}$, we have
\begin{equation}\label{2.3.3}
\int_{E}\big|[\Pi\vec{b},I_{\alpha}](f_{1},f_{2})(x)\big|^{q}dx\leq \frac{\gamma_{3}^{q}}{4^{q}}.
\end{equation}
\end{lemma}

\begin{proof}
Since $\int_{Q}\big(b(y)-b_{Q}\big)dy=0,$ it is easy to check that $f_{i}$ satisfies
$$\mathrm{supp} f_{i}\subset Q,$$
$$f_{i}(y_{i})(b(y)-b_{Q})\geq 0,$$
$$\int f_{i}(y_{i})dy_{i}=0,$$
$$|f_{i}(y_{i})|\leq 2|Q|^{-1/p_{i}},$$
$$\|f_{i}\|_{L^{p_{i}}}\leq 2,$$
$$\int \big(b(y)-b_{Q}\big)f_{i}dy=|Q|^{-1/p_{i}}\int_{Q}|b(y_{i})-b_{Q}|dy.$$

For a cube $Q$ with center $x_{Q}$ and $x\in (2\sqrt{n}Q)^{c}$, the following point-wise estimates hold:
\begin{equation}\label{2.3.4}
|I_{\alpha}\big((b-b_{Q})f_{1},(b-b_{Q})f_{2}\big)(x)|\lesssim |Q|^{\frac{1}{p'_{1}}+\frac{1}{p'_{2}}}|x-x_{Q}|^{-2n+\alpha},
\end{equation}
\begin{equation}\label{2.3.5}
|I_{\alpha}\big((b-b_{Q})f_{1},(b-b_{Q})f_{2}\big)(x)|\gtrsim \epsilon |Q|^{\frac{1}{p'_{1}}+\frac{1}{p'_{2}}}|x-x_{Q}|^{-2n+\alpha},
\end{equation}
\begin{equation}\label{2.3.6}
|I_{\alpha}\big((b-b_{Q})f_{1},f_{2}\big)(x)|\lesssim |Q|^{\frac{1}{p'_{1}}+\frac{1}{p'_{2}}+\frac{1}{n}}|x-x_{Q}|^{-2n+\alpha-1},
\end{equation}
\begin{equation}\label{2.3.7}
|I_{\alpha}\big(f_{1},(b-b_{Q})f_{2}\big)(x)|\lesssim |Q|^{\frac{1}{p'_{1}}+\frac{1}{p'_{2}}+\frac{1}{n}}|x-x_{Q}|^{-2n+\alpha-1},
\end{equation}
\begin{equation}\label{2.3.8}
|I_{\alpha}\big(f_{1},f_{2}\big)(x)|\lesssim |Q|^{\frac{1}{p'_{1}}+\frac{1}{p'_{2}}+\frac{1}{n}}|x-x_{Q}|^{-2n+\alpha-1},
\end{equation}
where $f_{i}$ as above and the constants involved are independent of $b, f_{i}$ and $\epsilon$.

To prove (\ref{2.3.4}), from the fact that $\|b\|_{*}=1$ and $x\in (2\sqrt{n}Q)^{c}$, we have
\begin{eqnarray*}
&&|I_{\alpha}\big((b-b_{Q})f_{1},(b-b_{Q})f_{2}\big)(x)|\\
&&=\bigg|\int_{\mathbb{R}^n}\int_{\mathbb{R}^n}\frac{(b(y_{1})-b_{Q})(b(y_{2})-b_{Q})f_{1}(y_{1})f_{2}(y_{2})}
{\big(|x-y_{1}|^{2}+|x-y_{2}|^{2}\big)^{n-\alpha/2}}dy_{1}dy_{2}\bigg|\\
&&\lesssim |Q|^{-\frac{1}{p_{1}}-\frac{1}{p_{2}}}|x-x_{Q}|^{-2n+\alpha}\prod_{i=1}^{2}
\int_{Q}(b(y_{i})-b_{Q})f_{i}(y_{i})dy_{i}\\
&&\lesssim|Q|^{-\frac{1}{p_{1}}-\frac{1}{p_{2}}}|x-x_{Q}|^{-2n+\alpha}\prod_{i=1}^{2}\int_{Q}
|b(y_{i})-b_{Q}|dy_{i}\\
&&\lesssim |Q|^{\frac{1}{p'_{1}}+\frac{1}{p'_{2}}}|x-x_{Q}|^{-2n+\alpha}.
\end{eqnarray*}

For (\ref{2.3.5}), using that $\big(b(y_{i})-b_{Q}\big)f_{i}(y_{i})\geq 0$, we can compute
\begin{eqnarray*}
&&|I_{\alpha}\big((b-b_{Q})f_{1},(b-b_{Q})f_{2}\big)(x)|\\
&&=\bigg|\int_{\mathbb{R}^n}\int_{\mathbb{R}^n}\frac{(b(y_{1})-b_{Q})(b(y_{2})-b_{Q})f_{1}(y_{1})f_{2}(y_{2})}
{\big(|x-y_{1}|^{2}+|x-y_{2}|^{2}\big)^{n-\alpha/2}}dy_{1}dy_{2}\bigg|\\
&&\gtrsim |x-x_{Q}|^{-2n+\alpha}\prod_{i=1}^{2}\bigg|\int_{Q}\big(b(y_{i})-b_{Q}\big)f_{i}(y_{i})dy_{i}\bigg|\\
&&=|x-x_{Q}|^{-2n+\alpha}\prod_{i=1}^{2}|Q|^{-\frac{1}{p_{i}}}\int_{Q}\big|b(y_{i})-b_{Q}\big|dy_{i}\\
&&\gtrsim \epsilon |Q|^{\frac{1}{p'_{1}}+\frac{1}{p'_{2}}}|x-x_{Q}|^{-2n+\alpha}.
\end{eqnarray*}

For (\ref{2.3.6}), by the fact $|f_{2}(y_{2})|\leq 2|Q|^{-1/p_{2}}$ and $\int_{Q} f_{2}(y_{2})dy_{2}=0$, we can also estimate
\begin{eqnarray*}
&&|I_{\alpha}\big((b-b_{Q})f_{1},f_{2}\big)(x)|\\
&&=\bigg|\int_{\mathbb{R}^n}\int_{\mathbb{R}^n}\frac{(b(y_{1})-b_{Q})f_{1}(y_{1})f_{2}(y_{2})}
{\big(|x-y_{1}|^{2}+|x-y_{2}|^{2}\big)^{n-\alpha/2}}dy_{1}dy_{2}\bigg|\\
&&=\bigg|\int_{\mathbb{R}^n}\int_{\mathbb{R}^n}\frac{(b(y_{1})-b_{Q})f_{1}(y_{1})f_{2}(y_{2})}
{\big(|x-y_{1}|^{2}+|x-y_{2}|^{2}\big)^{n-\alpha/2}}dy_{1}dy_{2}\\
&&\qquad -\int_{\mathbb{R}^n}\int_{\mathbb{R}^n}\frac{(b(y_{1})-b_{Q})f_{1}(y_{1})f_{2}(y_{2})}
{\big(|x-y_{1}|^{2}+|x-y'_{2}|^{2}\big)^{n-\alpha/2}}dy_{1}dy_{2}\bigg|\\
&&\lesssim |Q|^{-\frac{1}{p_{1}}-\frac{1}{p_{2}}}\int_{Q}
\int_{Q}\frac{|y_{2}-y'_{2}|(b(y_{1})-b_{Q})f_{1}(y_{1})}{\big(|x-y_{1}|+|x-y_{2}|\big)^{2n-\alpha+1}}dy_{1}dy_{2}\\
&&\lesssim|Q|^{-\frac{1}{p_{1}}+\frac{1}{p'_{2}}+\frac{1}{n}}|x-x_{Q}|^{-2n+\alpha-1}\int_{Q}
|b(y_{1})-b_{Q}|dy_{1}\\
&&\lesssim |Q|^{\frac{1}{p'_{1}}+\frac{1}{p'_{2}}+\frac{1}{n}}|x-x_{Q}|^{-2n+\alpha-1}.
\end{eqnarray*}

It is easy to see that $|I_{\alpha}\big((b-b_{Q})f_{1},f_{2}\big)(x)|=|I_{\alpha}\big(f_{1},(b-b_{Q})f_{2}\big)(x)|$, then (\ref{2.3.7}) holds.

Finally using that $f_{1}$ has mean zero we obtain (\ref{2.3.8}) as follows.
\begin{eqnarray*}
&&|I_{\alpha}\big(f_{1},f_{2}\big)(x)|\\
&&=\bigg|\int_{\mathbb{R}^n}\int_{\mathbb{R}^n}\frac{f_{1}(y_{1})f_{2}(y_{2})}
{\big(|x-y_{1}|^{2}+|x-y_{2}|^{2}\big)^{n-\alpha/2}}-\frac{f_{1}(y_{1})f_{2}(y_{2})}{\big(|x-y'_{1}|^{2}+|x-y_{2}|^{2}\big)^{n-\alpha/2}}dy_{1}dy_{2}\bigg|\\
&&\lesssim |Q|^{-\frac{1}{p_{1}}-\frac{1}{p_{2}}}\int_{Q}
\int_{Q}\frac{|y_{1}-y'_{1}||f_{1}(y_{1})||f_{2}(y_{2})|}{\big(|x-y_{1}|+|x-y_{2}|\big)^{2n-\alpha+1}}dy_{1}dy_{2}\\
&&\lesssim |Q|^{\frac{1}{p'_{1}}+\frac{1}{p'_{2}}+\frac{1}{n}}|x-x_{Q}|^{-2n+\alpha-1}.
\end{eqnarray*}

Now, we give the proofs of (\ref{2.3.1})-(\ref{2.3.3}). Taking $\nu>16,$ by (\ref{2.3.6}) we obtain
\begin{eqnarray*}
&&\bigg(\int_{|x-x_{Q}|>\nu r_{Q}}\big|(b(x)-b_{Q})I_{\alpha}((b-b_{Q})f_{1},f_{2})(x)\big|^{q}dx\bigg)^{\frac{1}{q}}\\
&&\leq C|Q|^{\frac{1}{p'_{1}}+\frac{1}{p'_{2}}+\frac{1}{n}}\sum_{s=\lfloor\log_{2}\nu\rfloor}^{\infty}
\bigg(\int_{2^{s}r_{Q}<|x-x_{Q}|<2^{s+1}r_{Q}}\frac{|b(x)-b_{Q}|^{q}}{|x-x_{Q}|^{q(2n-\alpha+1)}}dx\bigg)^{\frac{1}{q}}\\
&&\leq C|Q|^{\frac{1}{p'_{1}}+\frac{1}{p'_{2}}+\frac{1}{n}}\sum_{s=\lfloor\log_{2}\nu\rfloor}^{\infty}2^{-s(2n-\alpha+1)}|Q|^{-2+\frac{\alpha-1}{n}}
\bigg(\int_{2^{s}r_{Q}<|x-x_{Q}|<2^{s+1}d_{j}}|b(x)-b_{Q}|^{q}dx\bigg)^{\frac{1}{q}}\\
&&\leq C\sum_{s=\lfloor\log_{2}\nu\rfloor}^{\infty}s2^{-s(2n-\alpha+1-\frac{n}{q})}\\
&&\leq C\sum_{s=\lfloor\log_{2}\nu\rfloor}^{\infty}2^{-s(2n-\alpha-\frac{n}{q}+\frac{1}{2})}\\
&&\leq C\nu^{-(2n-\alpha-\frac{n}{q}+\frac{1}{2})},
\end{eqnarray*}
where we have used that $s\leq 2^{s/2}$ for $4\leq \lfloor\log_{2}\nu\rfloor\leq s$.

Similarly, we also have
$$\bigg(\int_{|x-x_{Q}|>\nu r_{Q}}\big|(b(x)-b_{Q})I_{\alpha}(f_{1},(b-b_{Q})f_{2})(x)\big|^{q}dx\bigg)^{\frac{1}{q}}\leq C\nu^{-2n-\alpha+1-\frac{n}{q}},$$
$$\bigg(\int_{|x-x_{Q}|>\nu r_{Q}}\big|(b(x)-b_{Q})^{2}I_{\alpha}(f_{1},f_{2})(x)\big|^{q}dx\bigg)^{\frac{1}{q}}\leq C\nu^{-2n-\alpha+1-\frac{n}{q}}.$$

Then for $\mu>\nu,$ using (\ref{2.3.4}), (\ref{2.3.5}) and the estimates above, we get
\begin{eqnarray*}
&&\int_{\nu r_{Q}<|x-x_{Q}|<\mu r_{Q}}\big|[\Pi\vec{b},I_{\alpha}](f_{1},f_{2})(x)\big|^{q}dx\\
&&\geq C\bigg(\big|I_{\alpha}((b-b_{Q})f_{1},(b-b_{Q})f_{2})(x)\big|^{q}dx\bigg)^{1/q}\\
&&\qquad -C\bigg(\int_{\nu r_{Q}<|x-x_{Q}|}\big|(b(x)-b_{Q})I_{\alpha}((b-b_{Q})f_{1},f_{2})(x)\big|^{q}dx\bigg)^{1/q}\\
&&\qquad -C\bigg(\int_{\nu r_{Q}<|x-x_{Q}|}\big|(b(x)-b_{Q})I_{\alpha}(f_{1},(b-b_{Q})f_{2})(x)\big|^{q}dx\bigg)^{1/q}\\
&&\qquad -C\bigg(\int_{\nu r_{Q}<|x-x_{Q}|}\big|(b(x)-b_{Q})^{2}I_{\alpha}(f_{1},f_{2})(x)\big|^{q}dx\bigg)^{1/q}\\
&&\geq C\epsilon|Q|^{\frac{1}{p'_{1}}+\frac{1}{p'_{2}}}\bigg(\int_{\nu r_{Q}<|x-x_{Q}|<\mu r_{Q}}|x-x_{Q}|^{q(-2n+\alpha)}dx\bigg)^{1/q}-C\nu^{-2n+\alpha+\frac{n}{q}-\frac{1}{2}}\\
&&\geq C\epsilon \Big(\nu^{-2nq+n+\alpha q}-\mu^{-2nq+n+\alpha q}\Big)^{\frac{1}{q}}-C\nu^{-2n+\alpha+\frac{n}{q}-\frac{1}{2}}.
\end{eqnarray*}
We can select $\gamma_{1},\gamma_{2}$ in place of $\nu,\mu$ with $\gamma_{2}>>\gamma_{1}$, then (\ref{2.3.1}) and (\ref{2.3.2}) are verified for some $\gamma_{3}>0$.

We now verified (\ref{2.3.3}). Let $E\subset \big\{\gamma_{1}r_{Q}<|x-x_{Q}|<\gamma_{2}r_{Q}\big\}$ be an arbitrary measurable set. It follows from Minkowski inequality that
\begin{eqnarray*}
&&\bigg(\int_{E}\big|[\Pi\vec{b},I_{\alpha}](f_{1},f_{2})(x)\big|^{q}dx\bigg)^{\frac{1}{q}}\\
&&\lesssim |Q|^{\frac{1}{p'_{1}}+\frac{1}{p'_{2}}}\bigg(\int_{E}|x-x_{Q}|^{-q(2n-\alpha)}dx\bigg)^{\frac{1}{q}}
+|Q|^{\frac{1}{p'_{1}}+\frac{1}{p'_{2}}+\frac{1}{n}}\bigg(\int_{E}\frac{|b(x)-b_{Q}|^{q}}{|x-x_{Q}|^{q(2n-\alpha+1)}}dx\bigg)^{\frac{1}{q}}\\
&&\qquad +|Q|^{\frac{1}{p'_{1}}+\frac{1}{p'_{2}}+\frac{1}{n}}\bigg(\int_{E}\frac{|b(x)-b_{Q}|^{2q}}{|x-x_{Q}|^{q(2n-\alpha+1)}}dx\bigg)^{\frac{1}{q}}\\
&&\lesssim \bigg(\frac{|E|^{1/q}}{|Q|^{1/q}}+\Big(\frac{1}{|Q|}\int_{E}|b(x)-b_{Q}|^{q}dx\Big)^{\frac{1}{q}}
+\Big(\frac{1}{|Q|}\int_{E}|b(x)-b_{Q}|^{2q}dx\Big)^{\frac{1}{q}}\bigg)
\end{eqnarray*}
The same estimate as Lemma \ref{lem2} and taking $0<\beta<\min\{\tilde{C}^{1/n},\gamma_{2}\}$, we can obtain the desired result.
\end{proof}

In the proof of the boundedness of the iterated commutators, the following two important properties of the weights we will be using.
\begin{lemma}(\cite{CT})\label{lem4}
Let $1<p_{1},p_{2}<\infty, P=(p_{1},p_{2}), 0<\alpha<2n,\frac{\alpha}{n}<\frac{1}{p_{1}}+\frac{1}{p_{2}}=\frac{1}{p}$ and $q$ such that $\frac{1}{q}=\frac{1}{p_{1}}+\frac{1}{p_{2}}-\frac{\alpha}{n}$. Suppose that $\omega_{1}^{\frac{p_{1}q}{p}},\omega_{2}^{\frac{p_{2}q}{p}}\in A_{p}$. Then,
\begin{enumerate}
\item [\rm(i)] $\omega=(\omega_{1},\omega_{2})\in A_{P,q}$;
\item [\rm(ii)] $\mu_{\omega}=\omega^{q}_{1}\omega^{q}_{2}\in A_{p}\subset A_{q}$.
\end{enumerate}
\end{lemma}

In the proof of Theorem \ref{main2}, we need the following weighted version of the Frech\'{e}t-Kolmogorov-Riesz theorem. We refer to works by Hanche-Olsen and Holden \cite{HH} and Clop and Cruz \cite{CC}.

\begin{lemma}\label{lem5}
Let $1<q<\infty$ and $\omega\in A_{q}$. Suppose that the subset $\mathfrak{F}\subset L^{q}(\omega)$ satisfies the following conditions:

(i) norm boundedness uniformly
\begin{equation}\label{2.5.1}
\sup_{f\in \mathfrak{F}}\|f\|_{L^{q}(\omega)}<\infty;
\end{equation}

(ii) translation continuity uniformly
\begin{equation}\label{2.5.2}
\lim_{f\rightarrow 0}\|f(\cdot+y)-f(\cdot)\|_{L^{q}(\omega)}=0  \ uniformly\ in\ f\in \mathfrak{F};
\end{equation}

(iii) control uniformly away from the origin
\begin{equation}\label{2.5.3}
\lim_{A\rightarrow \infty}\int_{|x|>A}|f(x)|^{q}\omega(x)dx=0 \ uniformly\ in\ f\in \mathfrak{F};
\end{equation}
then $\mathfrak{F}$ is pre-compact in $L^{q}(\omega)$.
\end{lemma}

Another reduction in the proof of Theorem \ref{main2} will be made by slightly modifying
the bilinear fractional integral operator. This technique comes from Krantz and Li \cite{KL2} (see also \cite{B1}, \cite{CT}).
More precisely, for any $\delta>0$ small enough, the kernel $K^{\delta}(x,y_{1},y_{2})$ in $\mathbb{R}^{3n}$ such that for $\max\{|x-y_{1}|,|x-y_{2}|\}>2\delta$,
$$K^{\delta}(x,y_{1},y_{2})=\frac{1}{\big(|x-y_{1}|^{2}+|x-y_{2}|^{2}\big)^{n-\alpha/2}};$$

for $\max\{|x-y_{1}|,|x-y_{2}|\}<\delta$,
$$K^{\delta}(x,y_{1},y_{2})=0;$$

and for all multi-indexes with $|\gamma|\leq 1$,
$$\partial^{\gamma}K^{\delta}(x,y_{1},y_{2})\lesssim \frac{1}{\big(|x-y_{1}|+|x-y_{2}|\big)^{2n-\alpha+|\gamma|}}.$$

The operators $I_{\alpha}^{\delta}$ are defined by
$$I_{\alpha}^{\delta}(f_{1},f_{2})(x)=\int_{\mathbb{R}^n}\int_{\mathbb{R}^n}K^{\delta}(x,y_{1},y_{2})f_{1}(y_{1})f_{2}(y_{2})dy_{1}dy_{2}.$$

\begin{lemma}\label{lem6}
If $b\in C^{\infty}_{c}$ and $\mathbf{\omega}\in \mathbf{A_{P,q}}$, then
$$\lim_{\delta}\big\|[\Pi\vec{b},I_{\alpha}^{\delta}]-[\Pi\vec{b},I_{\alpha}]\big\|
_{L^{p_{1}}(\omega_{1}^{p_{1}})\times L^{p_{1}}(\omega_{2}^{p_{2}})\rightarrow L^{q}(\mu_{\omega})}=0.$$
\end{lemma}
\begin{proof}
The proof of this result is very similar to that of Lemma 2.1 in \cite{B1} and we omit the details.
\end{proof}

\section{Proofs of Theorem \ref{main1} and Theorem \ref{main2}}

\vskip 0.5cm
\noindent
{\it Proof of Theorem \ref{main1}.}
We first show that if $[b,B_{\alpha}]_{1}$ is bounded from $L^{p_{1}}\times L^{p_{2}}$ to $L^{p}$, then $b\in {\rm BMO}$.
For $z_{0}\in \mathbb{R}^{n}\backslash \{0\}$, let $\delta=\frac{|z_{0}|}{2\sqrt{n}}$ and $Q_{0}(z_{0},\delta)$ denote the open cube centered at $z_{0}$ with side length $2\delta$. Then $|x|^{n-\alpha}$ has an absolutely convergent Fourier series
$$|x|^{n-\alpha}=\sum a_{m}e^{i v_{m}\cdot x}$$
with $\sum |a_{m}|<\infty$, where the exact form of the vectors $v_{m}$ is unrelated. Taking $z_{1}=\frac{z_{0}}{\delta}$ we have the expansion
$$|x|^{n-\alpha}=\delta^{-n+\alpha}|\delta x|^{n-\alpha}=\delta^{-n+\alpha}\sum a_{m}e^{i v_{m} \cdot \delta x} ~ for ~ |x-z_{1}|<\sqrt{n}.$$
Given cubes $Q=Q(x_{0},r)$ and $Q'=Q(x_{0}-rz_{1},r)$, if $x\in Q$ and $y\in Q'$, then
$$\Big|\frac{x-y}{r}-\frac{z_{0}}{\delta}\Big|\leq \Big|\frac{x-x_{0}}{r}\Big|+\Big|\frac{y-(x_{0}-\frac{r z_{0}}{\delta})}{r}\Big|< \sqrt{n},$$
and
$$|2x-y-(x_{0}-rz_{1})|\leq |x-x_{0}|+|x-y-rz_{1}|\leq \frac{\sqrt{n}}{2}r+r\Big|\frac{x-y}{r}-\frac{z_{0}}{\delta}\Big|\leq \frac{3\sqrt{n}}{2}r,$$
which implies that $2x-y\in \tilde{Q}=3\sqrt{n}Q'$.

Let $s(x)=\overline{\mathrm{sgn}(\int_{Q'}(b(x)-b(y))dy)}$. Then
\begin{eqnarray*}
&&\frac{1}{|Q|}\int_{Q}|b(x)-b_{Q'}|dx\\
&&=\frac{1}{|Q|}\frac{1}{|Q'|}\int_{Q}\bigg|\int_{Q'}(b(x)-b(y))dy\bigg|dx\\
&&=\frac{1}{|Q|^{2}}\int_{Q}\int_{Q'}s(x)(b(x)-b(y))dydx\\
&&=\frac{1}{|Q|^{2}}\int_{Q}\int_{Q'}\frac{r^{n-\alpha}s(x)(b(x)-b(y))}{|x-y|^{n-\alpha}}\big|\frac{x-y}{r}\big|^{n-\alpha}dydx\\
&&=\frac{1}{|Q|}\int_{Q}\int_{Q'}\frac{s(x)(b(x)-b(y))}{|x-y|^{n-\alpha}}\sum a_{m}e^{i v_{m} \cdot \delta\frac{x-y}{r}}dydx\\
&&=\frac{1}{|Q|}\int_{Q}\int_{Q'}\frac{s(x)(b(x)-b(y))}{|x-y|^{n-\alpha}}\sum a_{m}e^{i v_{m}\cdot \delta\frac{2x-y}{r}}e^{-i v_{m}\cdot \delta\frac{y}{r}}dydx\\
&&=\frac{1}{|Q|}\int_{Q}\int_{Q'}\frac{s(x)(b(x)-b(y))}{|x-y|^{n-\alpha}}\sum a_{m}e^{i v_{m}\cdot \delta\frac{2x-y}{r}}e^{-i v_{m} \cdot \delta\frac{y}{r}}\chi_{Q}(x)\chi_{Q'}(y)\chi_{\tilde{Q}}(2x-y)dydx
\end{eqnarray*}
Setting
$$f_{m}(y)=e^{-i v_{m}\cdot\frac{\delta}{r}y}\chi_{Q'}(y),$$
$$g_{m}(z)=e^{-i v_{m}\cdot\frac{\delta}{r}z}\chi_{\tilde{Q}}(z),$$
$$h_{m}(x)=e^{i v_{m}\cdot\frac{\delta}{r}x}s(x)\chi_{Q}(x),$$
we have
\begin{eqnarray*}
\frac{1}{|Q|}\int_{Q}|b(x)-b_{Q'}|dx&=&\frac{1}{|Q|}\sum a_{m}\int_{\mathbb{R}^{n}}\int_{\mathbb{R}^{n}}\frac{(b(x)-b(y))f_{m}(y)g_{m}(2x-y)}{|x-y|^{n-\alpha}}h_{m}(x)dydx\\
&=&\frac{1}{|Q|}\sum a_{m}\int_{\mathbb{R}^{n}}[b,B_{\alpha}]_{1}(f_{m},g_{m})(x)h_{m}(x)dx.
\end{eqnarray*}
It follows that
\begin{eqnarray*}
\frac{1}{|Q|}\int_{Q}|b(x)-b_{Q}|dx&\lesssim& \frac{1}{|Q|}\int_{Q}|b(x)-b_{Q'}|dx\\
&\lesssim&\frac{1}{|Q|^{1/q}}\sum |a_{m}|\|[b,B_{\alpha}]_{1}(f_{m},g_{m})\|_{L^{q}}\\
&\lesssim&\frac{\|f_{m}\|_{L^{p_{1}}}\|g_{m}\|_{L^p_{2}}}{|Q|^{1/p}}\|[b,B_{\alpha}]_{1}\|_{L^{p_{1}}\times L^{p_{2}}\rightarrow L^{q}}\\
&\lesssim&\|[b,B_{\alpha}]_{1}\|_{L^{p_{1}}\times L^{p_{2}}\rightarrow L^{q}},
\end{eqnarray*}
which yields $b\in \mathrm{BMO}$ and $\|b\|_{*}\lesssim C\|[b,B_{\alpha}]_{1}\|_{L^{p_{1}}\times L^{p_{2}}\rightarrow L^{q}},$.

To prove $b$ be an element of ${\rm CMO}$, we will adapt some arguments from \cite{CDW1}, see also \cite{CT}, which in turn are based on the original work in \cite{U}. The approach is the following: if one of the conditions Eqs.(\ref{2.1.1})-(\ref{2.1.3}) in Lemma \ref{lem1} is failed, we will show that there exist sequences of functions, $\{f_{j}\}_{j}$ uniformly bounded on $L^{p_{1}}$ and $\{g_{j}\}_{j}$ uniformly bounded on $L^{p_{2}}$, such that $[b,B_{\alpha}]_{1}(f_{j},g_{j})$ has no convergent subsequence, which contradicts the assumption that $[b,B_{\alpha}]_{1}$ is compact. It gives us that if $[b,B_{\alpha}]_{1}$ is compact, $b$ must satisfy all three conditions; that is $b\in {\rm CMO}$.

By Lemma \ref{lem2}, it is sufficient to once again repeat the steps preformed in \cite{CDW1} (or \cite{CT}, \cite{CDW2}) to obtain the desired result. Hence, the proof of Theorem \ref{main1} is completed. \qed

\vskip 0.5cm
\noindent
{\it Proof of Theorem \ref{main2}.}
$(A1)\Rightarrow (A2)$: Note that (\ref{2.5.1}) is immediate since for $b\in C_{c}^{\infty}$, $[b,I^{\delta}_{\alpha}]$ is bounded from $L^{p_{1}}(\omega_{1}^{p_{1}})\times L^{p_{2}}(\omega^{p_{2}}_{2})$ to $L^{q}(\mu_{\omega})$, because $\omega\in A_{P,q}$ by Lemma \ref{lem6}.

To show that (\ref{2.5.2}) holds, we can use a similar method as in \cite[p.491]{CT}, the proof of this results is very similar to that of linear commutator case, we omit the detail.

Now we give the estimate for (\ref{2.5.3}). Denote
$$\mathfrak{F}=\big\{[\Pi\vec{b},I^{\delta}_{\alpha}](f_{1},f_{2}): f_{i}\in L^{p_{i}}(\omega^{p_{i}}_{i}), \|f_{i}\|_{L^{p_{i}}(\omega_{i})} \leq 1,i=1,2 \big\}.$$
Then $\mathfrak{F}$ is uniformly bounded, because $[\Pi\vec{b},I_{\alpha}^{\delta}]$ is a bounded operator from $L^{p_{1}}(\omega_{1}^{p_{1}})\times L^{p_{2}}(\omega^{p_{2}}_{2})$ to $L^{q}(\mu_{\omega})$, because $\omega\in A_{P,q}$ by Lemma 2.1. To prove the uniform equicontinuity of $\mathfrak{F}$, we must see that
$$\lim_{t\rightarrow 0}\|[\Pi\vec{b},I_{\alpha}^{\delta}](f_{1},f_{2})(\cdot+t)-[\Pi\vec{b},I_{\alpha}^{\delta}](f_{1},f_{2})(\cdot)\|_{L^{q}(\mu_{\omega})}=0.$$

To do this, we write
\begin{eqnarray*}
&&[\Pi\vec{b},I_{\alpha}^{\delta}](f_{1},f_{2})(x)-[\Pi\vec{b},I_{\alpha}^{\delta}](f_{1},f_{2})(x+t)\\
&&=\int_{\mathbb{R}^n}\int_{\mathbb{R}^n}\big(b(x)-b(y_{1})\big)\big(b(x)-b(y_{2})\big)K^{\delta}(x,y_{1},y_{2})f_{1}(y_{1})f_{2}(y_{2})dy_{1}dy_{2}\\
&&\quad-\int_{\mathbb{R}^n}\int_{\mathbb{R}^n}\big(b(x+t)-b(y_{1})\big)\big(b(x+t)-b(y_{2})\big)K^{\delta}(x+t,y_{1},y_{2})f_{1}(y_{1})f_{2}(y_{2})dy_{1}dy_{2}\\
&&=\big(b(x)-b(x+t)\big)^{2}\int_{\mathbb{R}^n}\int_{\mathbb{R}^n}K^{\delta}(x,y_{1},y_{2})f_{1}(y_{1})f_{2}(y_{2})dy_{1}dy_{2}\\
&&\quad +\big(b(x)-b(x+t)\big)\int_{\mathbb{R}^n}\int_{\mathbb{R}^n}\big(b(x+t)-b(y_{2})\big)K^{\delta}(x,y_{1},y_{2})f_{1}(y_{1})f_{2}(y_{2})dy_{1}dy_{2}\\
&&\quad +\big(b(x)-b(x+t)\big)\int_{\mathbb{R}^n}\int_{\mathbb{R}^n}\big(b(x+t)-b(y_{1})\big)K^{\delta}(x,y_{1},y_{2})f_{1}(y_{1})f_{2}(y_{2})dy_{1}dy_{2}\\
&&\quad +\int_{\mathbb{R}^n}\int_{\mathbb{R}^n}\big(b(x+t)-b(y_{1})\big)\big(b(x+t)-b(y_{2})\big)\\
&&\qquad \qquad\times \big(K^{\delta}(x,y_{1},y_{2})-K^{\delta}(x+t,y_{1},y_{2})\big)f_{1}(y_{1})f_{2}(y_{2})dy_{1}dy_{2}\\
&&=I_{1}(x,t)+I_{2}(x,t)+I_{3}(x,t)+I_{4}(x,t).
\end{eqnarray*}

For $I_{1}(x,t)$, we simply have
$$|I_{1}(x,t)|\lesssim |t|^{2}\|\nabla b\|_{\infty}I_{\alpha}(|f_{1}|,|f_{2}|)(x),$$
which implies that
$$\|I_{1}(\cdot,t)\|_{L^{q}(\omega)}\lesssim |t|^{2}.$$

Similarly, we also have that for $j=2,3$
$$|I_{j}(x,t)|\lesssim |t|\|\nabla b\|_{\infty}\|b\|_{\infty}I_{\alpha}(|f_{1}|,|f_{2}|)(x),$$
and
$$\|I_{j}(\cdot,t)\|_{L^{q}(\omega)}\lesssim |t|.$$

We now give the estimate for $I_{4}(x,t)$. We may assume that $|t|\in (0,\delta/4)$. Thus, if $\max\{|x-y_{1}|,|x-y_{2}|\}\leq \delta/2$ we have
$$K^{\delta}(x+t,y_{1},y_{2})-K^{\delta}(x,y_{1},y_{2})=0$$
and $\max\{|x-y_{1}|,|x-y_{2}|\}> \delta/2$ we have
$$\max\{|x-y_{1}|,|x-y_{2}|\}> 2t.$$
This gives us that
\begin{eqnarray*}
\big|I_{4}(x,t)\big|&\lesssim& |t|\|b\|_{\infty}^{2}\int\int_{\max\{|x-y_{1}|,|x-y_{2}|\}> \delta/2}\frac{|f_{1}(y_{1})f_{2}(y_{2})|}{\big(|x-y_{1}|+|x-y_{2}|\big)^{2n-\alpha+1}}dy_{1}dy_{2}\\
&\lesssim& |t|\|b\|_{\infty}^{2}\sum_{j\geq 0}\int\int_{2^{j-1}\delta<\max\{|x-y_{1}|,|x-y_{2}|\}< 2^{j}\delta}\frac{|f_{1}(y_{1})f_{2}(y_{2})|}{\big(|x-y_{1}|+|x-y_{2}|\big)^{2n-\alpha+1}}dy_{1}dy_{2}\\
&\lesssim& |t|\|b\|_{\infty}^{2}\sum_{j\geq 0}\int\int_{2^{j-1}\delta<\max\{|x-y_{1}|,|x-y_{2}|\}< 2^{j}\delta}\frac{|f_{1}(y_{1})f_{2}(y_{2})|}{\big(|x-y_{1}|+|x-y_{2}|\big)^{2n-\alpha}}\\
&&\qquad \times\frac{1}{\max\{|x-y_{1}|,|x-y_{2}|\}}dy_{1}dy_{2}\\
&\lesssim&\frac{|t|}{\delta}\|b\|_{\infty}^{2}I_{\alpha}(|f_{1}|,|f_{2}|)(x),
\end{eqnarray*}
which gives
$$\|I_{4}(\cdot,t)\|_{L^{q}(\mu_{\omega})}\leq |t|.$$

Combining the estimates above and let $|t|<1$, we have
$$\|[\Pi\vec{b},I_{\alpha}^{\delta}](f_{1},f_{2})(\cdot+t)-[\Pi\vec{b},I_{\alpha}^{\delta}](f_{1},f_{2})(\cdot)\|_{L^{q}(\mu_{\omega})}\lesssim |t|.$$

Thus, if $b\in {\rm CMO}$, $[\Pi\vec{b},I_{\alpha}]$ is a compact operator from $L^{p_{1}}(\omega_{1}^{p_{1}})\times L^{p_{2}}(\omega_{2}^{p_{2}})$ to $L^{q}(\mu_{\omega})$.

Obviously $(A2)\Rightarrow (A3)$. So it remains to show that $(A3)\Rightarrow (A1)$. By Lemma \ref{lem3} and the same argument as Theorem \ref{main1}, we need only to prove $b\in {\rm BMO}$.

$(A3)\Rightarrow (A1)$: Let $z_{0}\in \mathbb{R}^n$ such that $|(z_{0},z_{0})|>2\sqrt{n}$ and let $\delta$ small enough such that $\delta<1$. Take $B=B\big((z_{0},z_{0}),\delta\sqrt{2n}\big)\subset \mathbb{R}^{2n}$ be the ball for which we can express $(|y_{1}|^{2}+|y_{2}|^{2})^{n-\alpha/2}$ as an absolutely convergent Fourier series of the form
$$(|y_{1}|^{2}+|y_{2}|^{2})^{n-\alpha/2}=\sum_{j}a_{j}e^{iv_{j}\cdot(y_{1},y_{2})}, \quad (y_{1},y_{2})\in B,$$
where $\sum_{j}|a_{j}|<\infty$ and we do not care about the vectors $v_{j}\in \mathbb{R}^{2n},$ but we will at times express them as $v_{j}=(v_{j}^{1},v_{j}^{2})\in \mathbb{R}^{n}\times \mathbb{R}^n.$

Set $z_{1}=\delta^{-1}z_{0}$ and note that
$$\big(|y_{1}-z_{1}|^{2}+|y_{2}-z_{1}|^{2}\big)^{1/2}<\sqrt{2n}\Rightarrow \big(|\delta y_{1}-z_{0}|^{2}+|\delta y_{2}-z_{0}|^{2}\big)^{1/2}<\delta \sqrt{2n}.$$
Then for any $(y_{1},y_{2})$ satisfying the inequality on the left, we have
$$(|y_{1}|^{2}+|y_{2}|^{2})^{n-\alpha/2}=\delta^{-2n+\alpha}(|\delta y_{1}|^{2}+|\delta y_{2}|^{2})^{n-\alpha/2}=\delta^{-2n+\alpha}\sum_{j}a_{j}e^{i\delta v_{j}\cdot(y_{1},y_{2})}.$$

Let $Q=Q(x_{0},r)$ be any arbitrary cube in $\mathbb{R}^n$. Set $\tilde{z}=x_{0}+rz_{1}$ and take $Q'=Q(\tilde{z},r)\subset \mathbb{R}^n$. So for any $x\in Q$ and $y_{1},y_{2}\in Q'$, we have
$$\Big|\frac{x-y_{1}}{r}-z_{1}\Big|\leq \Big|\frac{x-x_{0}}{r}\Big|+\Big|\frac{y_{1}-\tilde{z}}{r}\Big|\leq \sqrt{n},
\quad \Big|\frac{x-y_{2}}{r}-z_{1}\Big|\leq \Big|\frac{x-x_{0}}{r}\Big|+\Big|\frac{y_{2}-\tilde{z}}{r}\Big|\leq \sqrt{n},$$
which implies that
$$\bigg(\Big|\frac{x-y_{1}}{r}-z_{1}\Big|^{2}+\Big|\frac{x-y_{2}}{r}-z_{1}\Big|^{2}\bigg)^{1/2}\leq \sqrt{2n}.$$

Let $s(x)=\overline{\mathrm{sgn}(\int_{Q'}(b(x)-b(y))dy)}$. We have the following estimate,
\begin{eqnarray*}
&&\Big(\frac{1}{|Q|}\int_{Q}|b(x)-b_{Q}|dx\Big)^{2}\lesssim\Big(\frac{1}{|Q|}\int_{Q}|b(x)-b_{Q'}|dx\big)^{2}\\
&&\lesssim\frac{1}{|Q|}\int_{Q}|b(x)-b_{Q'}|^{2}dx\lesssim \frac{1}{|Q|}\int_{Q}s(x)^{2}(b(x)-b_{Q'})^{2}dx\\
&&\lesssim \frac{1}{|Q||Q'|^{2}}\int_{Q}\int_{Q'}\int_{Q'}s(x)^{2}\big(b(x)-b(y_{1})\big)\big(b(x)-b(y_{2})\big)dy_{1}dy_{2}dx\\
&&\lesssim \frac{r^{2n-\alpha}\delta^{-2n+\alpha}}{|Q|^{3}}\int_{Q}\int_{Q'}\int_{Q'}\frac{s(x)^{2}\big(b(x)-b(y_{1})\big)\big(b(x)-b(y_{2})\big)}
{\big(|x-y_{1}|^{2}+|x-y_{2}|^{2}\big)^{n-\alpha/2}}\sum_{j}a_{j}e^{i\frac{\delta}{r}v_{j}\cdot(x-y_{1},x-y_{2})}dy_{1}dy_{2}dx.
\end{eqnarray*}
Setting
$$g_{j}(y_{1})=e^{-i\frac{\delta}{r}v^{1}_{j}\cdot y_{1}}\chi_{Q'}(y_{1}),$$
$$h_{j}(y_{2})=e^{-i\frac{\delta}{r}v^{2}_{j}\cdot y_{2}}\chi_{Q'}(y_{2}),$$
$$m_{j}(x)=e^{i\frac{\delta}{r}v_{j}\cdot (x,x)}\chi_{Q}(x)s(x)^{2}.$$
We have
\begin{eqnarray*}
&&\Big(\frac{1}{|Q|}\int_{Q}|b(x)-b_{Q}|dx\Big)^{2}\\
&&\lesssim \frac{r^{2n-\alpha}\delta^{-2n+\alpha}}{|Q|^{3}}\sum_{j}|a_{j}|\int_{\mathbb{R}^n}\big|[\Pi\vec{b},I_{\alpha}](g_{j},h_{j})(x)m_{j}(x)\big|dx\\
&&\lesssim r^{-n-\alpha}\delta^{-2n+\alpha}\sum_{j}|a_{j}|\|[\Pi\vec{b},I_{\alpha}]\|_{L^{p_{1}}\times L^{p_{2}}\rightarrow L^{q}}\|g_{j}\|_{L^{p_{1}}}\|h_{j}\|_{L^{p_{2}}}\|m_{j}\|_{L^{q'}}\\
&&\lesssim \delta^{-2n+\alpha}\|[\Pi\vec{b},I_{\alpha}]\|_{L^{p_{1}}\times L^{p_{2}}\rightarrow L^{q}}\sum_{j}|a_{j}|.
\end{eqnarray*}
The desired result follows from here. \qed

{\bf Acknowledgments} We would like to thank the anonymous referee for his/her comments.

\color{black}


\vskip 0.5cm
\begin{thebibliography}{999}

\bibitem{B1}
B\'{e}nyi, \'{A}., Dami\'{a}n, W., Moen, K., Torres, R.H.: Compactness properties of commutators of bilinear frctional integrals. \emph{Math. Z.}, (2015), doi: 10.1007/s00209-015-1437-4.

\bibitem{B2}
B\'{e}nyi, \'{A}., Dami\'{a}n, W., Moen, K., Torres, R.H.: Compact bilinear operators: the weighted case. Mich. Math. J. {\bf 4}, 39-51 (2015)

\bibitem{BT}
B\'{e}nyi, \'{A}., Torres, R.H.: Compact bilinear operators and commutators. Proc. Amer. Math. Soc. {\bf 141}(10), 3609-3621 (2013)

\bibitem{CT}
Chaffee, L., Torres, R.H.: Characterization of Compactness of the Commutators of Bilinear Fractional Integral Operators. Potential Anal. {\bf43}, 481-494 (2015)

\bibitem{CDW1}
Chen, Y., Ding, Y., Wang, X.: Compactness of commutators of Riesz potential on Morrey spaces. Potential Anal. {\bf 30}, 301-313(2009)

\bibitem{CDW2}
Chen, Y., Ding, Y., Wang, X.: Compactness of commutators for Singular integrals on Morrey spaces. Canad. J. Math. {\bf 64}, 257-281(2012)

\bibitem{CC}
Clop, A., Cruz, V.: Weighted estimates for Beltrami equations. Ann. Acad. Sci. Fenn. Math. {\bf 38}, 91-113 (2013)

\bibitem{CRW}
Coifman, R., Rochberg, R., Weiss,G.: Factorization theorems for Hardy
spaces in several variables. \emph{Ann. of Math.}, {\bf103}, 611-635 (1976)

\bibitem{G}
Grafakos, L.: On multilinear fractional integrals. Studia Math. {\bf102}, 49-56 (1992)

\bibitem{GK}
L. Grafakos, L., Kalton, N.: Some remarks on multilinear maps and interpolation, Math. Ann.
{\bf 319}, 151-180 (2001)

\bibitem{HH}
Hanche-Olsen, H., Holden, H.: The Kolmogorov-Riesz compactness theorem. Expo. Math. {\bf 28}, 385-394 (2010)

\bibitem{KS}
Kenig, C., Stein, E.: Multilinear estimates and fractional integration, Math. Res. Lett. {\bf 6}
1-15 (1999)

\bibitem{KL1}
Krantz, S., Li, S.-Y.: Boundedness and compactness of integral operators on spaces of homogeneous
type and applications, I. J. Math. Anal. Appl. {\bf 258}, 629-641 (2001)

\bibitem{KL2}
Krantz, S., Li, S.-Y.: Boundedness and compactness of integral operators on spaces of homogeneous
type and applications, II. J. Math. Anal. Appl. {\bf 258}, 642-657 (2001)

\bibitem{LOPTT}
Lerner, A.K., Ombrosi, S., P\'{e}rez, C., Torres, R.H., Trujillo-Gonz\'{a}lez, R.: New maximal functions and multiple weights for the multilinear Calder\'{o}n-Zygmund theory. \emph{Adv. Math.}, {\bf220}, 1222-1264 (2009)

\bibitem{M}
Moen, K.: Weighted inequalities for multilinear fractional integral operators. Collect Math. {\bf 60}, 213-238 (2009)

\bibitem{P}
P\'{e}rez, C.: Endpoint estimates for commutators of singular integral operators. \emph{J. Funct. Anal.}, {\bf128}, 163-185
(1995)

\bibitem{PT}
P\'{e}rez, C., Torres, R.H.: Sharp maximal function estimates for multilinear singular integrals. \emph{Contemp. Math.}, {\bf 320}, 323-331 (2003)

\bibitem{S}
Stein, E.M.: Singular Integral and Differentiability Properties of Functions. Princeton University
Press, Princeton (1971)

\bibitem{T}
Tang, L.: Weighted estimates for vector-valued commutators of multilinear operators. \emph{Proc. Roy. Soc. Edinburgh Sect. A}, {\bf138}, 897-922 (2008)

\bibitem{U}
Uchiyama, A.: On the compactness of operators of Hankel type. T\^{o}hoku Math. J. {\bf 30}, 163-171(1978)

\bibitem{W}
Wang, S.: The compactness of the commutator of fractional integral operator (in Chinese). Chin. Ann. Math {\bf 8}(A), 475-482(1987)

\end{thebibliography}
\end{document}